%% file: cdc2020_arxiv.tex
\documentclass[letterpaper, 10 pt, conference]{ieeeconf}

\IEEEoverridecommandlockouts

\usepackage[utf8]{inputenc}
\usepackage{tikz}
\usetikzlibrary{arrows} 
\usepackage{amsmath,amsfonts,amssymb,bm,accents,mathtools,dsfont}
\usepackage{graphicx,cite,enumerate,hyperref,url,soul}
\usepackage{blindtext}
\graphicspath{{Images/}}

\usepackage{color}

\input{mysymbol.tex}

\newcommand{\includesvg}[2][scale=1]{\includegraphics[#1]{#2.pdf}}

\usepackage[english]{babel}

\usepackage{algpseudocode,algorithm}
\algrenewcommand\algorithmicdo{}
\algrenewtext{EndFor}{\algorithmicend}
\algrenewtext{EndProcedure}{\algorithmicend}
\algrenewtext{EndWhile}{\algorithmicend}

\newtheorem{proposition}{Proposition}

\newtheorem{assumption}{Assumption}

\title{\LARGE \bf Resilient Control: Compromising to Adapt}

\author{Luiz~F.~O.~Chamon, Alexandre Amice, Santiago Paternain, and Alejandro~Ribeiro%
\thanks{Department of Electrical and Systems Engineering, University of Pennsylvania, Philadelphia, USA.
\texttt{\small \{luizf, amice, spater, aribeiro\}@seas.upenn.edu} This work
was supported by ARL DCIST CRA W911NF-17-2-0181.%
}}

\begin{document}
\maketitle

\begin{abstract}

In optimal control problems, disturbances are typically dealt with using robust solutions, such as~$\calH_\infty$ or tube model predictive control, that plan control actions feasible for the worst-case disturbance. Yet, planning for every contingency can lead to over-conservative, poorly performing solutions or even, in extreme cases, to infeasibility. Resilience addresses these shortcomings by adapting the underlying control problem, e.g., by relaxing its specifications, to obtain a feasible, possibly still valuable trajectory. Despite their different aspects, robustness and resilience are often conflated in the context of dynamical systems and control. The goal of this paper is to formalize, in the context of optimal control, the concept of resilience understood as above, i.e., in terms of adaptation. To do so, we introduce a resilient formulation of optimal control by allowing disruption-dependent modifications of the requirements that induce the desired resilient behavior. We then propose a framework to design these behaviors automatically by trading off control performance and requirement violations. We analyze this resilience-by-compromise method to obtain inverse optimality results and quantify the effect of disturbances on the induced requirement relaxations. By proving that robustness and resilience optimize different objectives, we show that these are in fact distinct system  properties. We conclude by illustrating the effect of resilience in different control problems.

\end{abstract}

\section{INTRODUCTION}
	\label{S:intro}

Coping with disruptions is a core requirement for autonomous systems operating in the real world. Indeed, as these complex systems leave the controlled setting of the lab, it becomes increasingly important to enable them to safely negotiate adverse situations arising from the dynamic and fast-evolving environments in which they must operate~\cite{Gunes14a, Stankovic05o}. In the context of dynamical systems and control, this issue is often addressed through the concept of \emph{robustness}. The robust approach plans for the worst so that the resulting system can achieve its objective~(e.g., state regulation) regardless of the conditions in which it operates. Techniques such as~$\calH_\infty$ control, tube model predictive control~(MPC), and robust system-level synthesis have been developed specifically to address this issue~\cite{Dullerud13a, Li00r, Borrelli17p, Anderson19s}. In simple terms, robust systems are ``hard to break.''

Yet, the success of robustness may also be the root of its shortcomings. It is often not viable to plan for every contingency as it would lead to over-conservative behaviors whose performance is deficient even under normal operating conditions. In extreme cases, the resulting control problem may simply be infeasible. Hence, the question is no longer how to operate under or deal with a certain level of disturbance, but what to do when things go so catastrophically wrong that the original equilibrium is no longer viable. In such cases, the only solution is to modify the system requirements, e.g., by removing unlikely contingencies or relaxing specifications, to find an alternative equilibrium.

In ecology, this capacity of systems to adapt and recover from disruptions by modifying their underlying operation is known as \emph{resilience}~\cite{Holling73r, Holling96e}. Since its introduction in the 1970s, it has been observed in a myriad of ecosystems and incorporated in fields such as psychology and dynamical/cyber-physical systems~\cite{Werner89v, Rodin14t, Rieger09r, Zhu15game, Ramachandran19r}. Contrary to stability, characterized by the persistence of a system near an equilibrium, resilience emphasizes conditions far from steady state, where instabilities can flip a system into another behavior regime~\cite{Holling96e}. In simple terms, resilient systems are ``easy to fix.''

In dynamical systems and control, robustness and resilience are often conflated. Even when resilience is described, the sought after behaviors are often robust in the sense of the above definitions, e.g.,~\cite{Chen18r, Tzoumas17r, Guerrero17f}. Even in his seminal works, Holling discriminates between ``engineering resilience''~(robustness) and ``ecological resilience,'' by distinguishing systems with a single equilibrium from those with multiple equilibria~\cite{Holling96e}. Though resilient solutions involving adaptation to disruptions have been studied, such as in~\cite{Gunes14a, Stankovic05o, Zhu15game, Ramachandran19r}, a formal, general definition of resilient control akin to its robust counterpart is still lacking.

The goal of this work is to formalize resilience in the context of optimal control. We begin by introducing the general problem of constrained control under disturbances and its robust solution~(Section~\ref{S:prob}). We then formulate the resilient optimal control problem by allowing controlled constraint violations in optimal control problems~(Section~\ref{S:resilience}). To be useful, however, these violations must be appropriately designed, which cannot be done manually for any moderately-sized problem. To address this issue, we put forward a framework to obtain requirement modifications by trading off control performance and violation costs. We analyze this formulation to obtain inverse optimality results and quantify the effect of disturbances on the violations. By proving that robustness and resilience optimize different objectives, we show that they are complementary properties that in many applications, may be simultaneously required~(Section~\ref{S:resilience_theory}). We conclude by deriving a practical algorithm to solve resilient control problems~(Section~\ref{S:algorithm}) and illustrating its use~(Section~\ref{S:sims}).

\section{PROBLEM FORMULATION}
	\label{S:prob}

Let~$\bXi$ be a random variable taking values in a compact set~$\calK \subseteq \setR^d$ according to some measure~$\mathfrak{p}$. We assume for simplicity that~$\mathfrak{p}$ is absolutely continuous with respect to the Lebesgue measure, so that~$\bXi$ has a probability density function~(Radon-Nikodym derivative) denoted~$f_{\bXi}$. Its realizations~$\bxi$ denote states of the world that may be construed as disturbances to the normal operation of an autonomous system represented by the prototypical constrained optimal control problem
\begin{prob}\label{P:generic}
	P^\star(\bXi) = \min_{\bz \in \setR^p}& &&J(\bz)
	\\
	\subjectto& &&g_i(\bz, \bXi) \leq 0
		\text{,} \quad i = 1,\dots,m
		\text{,}
\end{prob}
where~$\bz$ denotes the decision variable, e.g., actuation strength, $J$ is a control performance measure, and the~$g_i(\cdot,\bxi)$ describe the control requirements under~$\bxi$.

\begin{assumption}\label{A:convexity}
	The control performance~$J: \setR^p \to \setR$ is a strongly convex, continuously differentiable function, $g_i(\bz, \cdot) \in L_2$ are~$L_i$-Lipschitz continuous with respect to the~$\ell_\infty$-norm for all~$\bz \in \setR^p$, and~$g_i(\cdot, \bxi)$ are coercive~(radially unbounded), convex functions for all~$\bxi \in \calK$. The requirement functions~$g_i$ have continuous derivatives with respect to~$\bz$ and~$\bxi$.
\end{assumption}

Note that since~\eqref{P:generic} is parameterized by a random variable, its optimal solution~$\bz^\star(\bXi)$ and value~$P^\star(\bXi)$ are random and depend on the \emph{a priori} unknown disturbance realization. \emph{Our goal is to obtain a deterministic~$\bz^\dagger$ that is feasible for most~(if not all) realizations~$\bxi$ and whose performance~$P^\dagger = J(\bz^\dagger)$ is similar to the optimal~$P^\star(\bxi)$.} Though the latter objective is less critical, it is certainly desired.

To illustrate the use of~\eqref{P:generic} in control, note that it can cast the following constrained LQR problem~\cite{Borrelli17p}:
\begin{prob}\label{P:lqr}
	\minimize_{\bx_k,\,\bu_k}& &&\bx_N^T \bP \bx_N + \sum_{k = 0}^{N-1} \bx_k^T \bQ \bx_k + \bu_k^T \bR \bu_k
	\\
	\subjectto& &&\abs{\bx_k} \leq \bxb
		\text{,} \quad
	\abs{\bu_k} \leq \bub - \bXi_{u,k}
		\text{,}
	\\
	&&&\bx_{k+1} = \bA \bx_k + \bB \bu_k + \bXi_{d,k}
	\text{,}
\end{prob}
where~$\bx_k$ and~$\bu_k$ are the state and control action at time~$k$, respectively, of a linear dynamical system described by the state-space matrices~$\bA$ and~$\bB$, $\bxb$ and~$\bub$ are bounds on the state and actions, and the initial state~$\bx_0$ is given. Here, $\bz$ collects the~$\{\bx_k,\bu_{k-1}\}$ for~$k = 1,\dots,N$. The disturbances in~\eqref{P:lqr} model changes in the dynamics~($\bXi_{d,k}$) and/or disruptions to the system's actuation capabilities~($\bXi_{u,k}$). Namely, a realization~$[\bxi_{u,k}]_i = [\bub]_i$ is equivalent to actuator~$i$ being unavailable at instant~$k$. Hence, while the abstract~\eqref{P:generic} is the object of study of this paper, we are ultimately interested in the control problems it represents, e.g., \eqref{P:lqr}.

In the control literature, a common approach to obtaining the desired~$\bz^\dagger$ is to use the robust formulation of~\eqref{P:generic}
\begin{prob}[\textup{P-RO}]\label{P:robust}
	P^\star_\text{Ro} = \min_{\bz \in \setR^p}& &&J(\bz)
	\\
	\subjectto& &&\Pr\left[ \bg(\bz, \bXi) \leq \zeros \right] \geq 1-\delta
		\text{,}
\end{prob}
where the probability is taken with respect to the distribution of~$\bXi$ and the requirements~$g_i$ are collected in the vector-valued function~$\bg$ for conciseness. The probability of violation parameter~$\delta \in [0,1]$ trades-off feasibility for control performance~\cite{Schwarm99c, Li00r, Borrelli17p}. From the additivity of measures, it is straightforward that reducing~$\delta$ reduces the feasibility set of~\eqref{P:robust}, which may increase the control cost. For~$\delta = 0$, the constraints in~\eqref{P:robust} reduce to the classical worst-case formulation of robustness, enforcing that~$\max_{\bxi \in \calK}\ g_i(\bz, \bxi) \leq 0$, i.e., that the solution is feasible for all possible conditions~$\bxi$~\cite{Dullerud13a}. Yet, these conditions can render the control problem infeasible or lead to solutions with impractical levels of performance. These issues are sometimes overcome by the statistical formulation in~\eqref{P:robust}. Under mild conditions, feasible solutions of~\eqref{P:robust} can be obtained using a deterministic optimization problem~\cite{Li00r, Borrelli17p}.

\begin{proposition}
Let~$\bXi$ be a sub-Gaussian random vector~(e.g., Gaussian or Bernoulli), i.e., $\E\left[ e^{\nu \bu^T \left( \bXi - \E[\bXi] \right)} \right] \leq e^{\nu^2 \sigma^2/2}$ for all~$\nu \in \setR$ and~$\bu \in \setR^d$ such that~$\norm{\bu} = 1$. Then, under Assumption~\ref{A:convexity}, the unique
\begin{prob}[$\widehat{\textup{P}}\textup{-RO}$]\label{P:equivalentRobust}
	\bzh_\text{Ro} = \argmin_{\bz \in \setR^p}& &&J(\bz)
	\\
	\subjectto& &&g_i(\bz, \E[\bXi]) \leq -\epsilon
		\text{,}
\end{prob}
with~$\epsilon = L \sigma \sqrt{2 \log(2 m d/\delta)}$, for~$L = \max_i L_i$, is \eqref{P:robust}-feasible. In particular, if~$\calK \subseteq [0,\bar{\xi}]^d$, then~$\sigma \leq \bar{\xi}/2$.
\end{proposition}

\begin{proof}
Recall that since~$J$ is strongly convex, the solution of~\eqref{P:equivalentRobust} is unique~\cite{Bertsekas09c}. The proof then follows by bounding~$\Pr \left[ \max_i g_i(\bz^\star_\text{Ro}, \bXi) \leq 0 \right]$ using concentration of measure~\cite{Ledoux01t}. From the Lipschitz continuity of~$g_i$ we get
\begin{align*}
	g_i(\bz^\dagger_\text{Ro}, \bxi)
		&\leq g_i(\bz^\dagger_\text{Ro}, \E[\bXi]) + L_i \norm{\bxi - \E[\bXi]}_\infty
	\\
	{}&\leq -\epsilon + L_i \norm{\bxi - \E[\bXi]}_\infty
		\text{.}
\end{align*}
Note that since~$g_i(\bz^\dagger_\text{Ro}, \E[\bXi]) \leq 0$ we care only about the positive tail of the Lipschitz inequality. To proceed, use the union bound and Hoeffding's inequality to obtain that
\begin{equation}\label{E:hoeffding}
	\Pr \left[ \max_i L_i \norm{\bxi - \E[\bXi]}_\infty \leq \epsilon \right]
		\geq 1 - 2 d \sum_{i = 1}^m \exp\left( \frac{-\epsilon^2}{2 L_i^2 \sigma^2} \right) 
		\text{.}
\end{equation}
Using~$\epsilon$ as in the hypothesis ensures that~\eqref{E:hoeffding} is greater than~$1-\delta$, thus concluding the proof.
\end{proof}

Robust controllers are often deployed in critical applications, such as industrial process control and security constrained power allocation~\cite{Borrelli17p, Capitanescu11s}. Nevertheless, their worst case approach has two shortcomings. First, too stringent requirements on the probability of failure~$\delta$ can result in an infeasible problem or render the solution of~\eqref{P:robust} useless in practice due to its poor performance even in favorable conditions. What is more, sensitive requirements~(i.e., large~$L_i$) lead to large~$\epsilon_i$ in~\eqref{P:equivalentRobust}, considerably reducing its feasible set. Though~\eqref{P:robust} may be feasible even if~\eqref{P:equivalentRobust} is not, obtaining a solution of the former is challenging without the latter except in special cases~\cite{Dullerud13a, Schwarm99c, Li00r, Borrelli17p}. Second, even if~\eqref{P:robust} is feasible and its solution has reasonable performance, the issue remains of what happens in the~$\delta$ portion of the realizations in which a stronger than anticipated disturbance occurs. Indeed, though robust autonomous systems make failures unlikely, they do not account for how the system fails once it does. Hence, though unlikely, failures can be catastrophic. Resilience overcomes these limitations by adapting the underlying optimal control problem to disruptions.

\section{RESILIENT CONTROL}
	\label{S:resilience}

In a parallel to the ecology literature, we define resilience in autonomous systems as \emph{the ability to adapt to, and possibly recover from, disruptions}. In particular, we are interested in dealing with disturbances so extreme that the original control problem becomes ineffective or infeasible. Where robust control would declare failure, resilient control attempts to remain operational by modifying the underlying control problem, reverting to an alternative trajectory that violates requirements in a controlled manner. In practice, this means that when a resilient system suffers a disastrous shock that jeopardizes its ability to solve its original task, it will adapt and modify its requirements in an attempt to at least partially salvage its mission. Resilience is therefore not a replacement for robustness, which may be the only sensible course of action for critical requirements, but a complementary set of behaviors that a control system can display.

\subsection{Resilient optimal control}
\label{S:resilient_control}

To operationalize the above definition of resilient dynamical system, we must embed the optimal control problem~\eqref{P:generic} with the ability to modify its requirements depending on the disruption suffered by the system. A natural way to do so is by associating a disturbance-dependent relaxation~$s_i: \calK \to \setR_+$, $s_i \in L_2$, to the $i$-th~requirement as in
\begin{prob}[\textup{P-RE}]\label{P:parametrized}
	P^\star_\text{Re}(\bs) = \min_{\bz \in \setR^p}& &&J(\bz)
	\\
	\subjectto& &&\bg(\bz, \bxi) \leq \bs(\bxi)
		\text{,} \quad \bxi \in \calK
		\text{,}
\end{prob}
where the vector-valued function~$\bs$ collects the slacks~$s_i$. Depending on~$\calK$, \eqref{P:parametrized} may have a finite or infinite number of constraints. The latter case can be tackled using semi-infinite programming algorithms~\cite{Reemtsen98s, Bonnans00p}.

The violations~$\bs(\bxi)$ in~\eqref{P:parametrized} determine how the underlying control problem is modified to adapt to the operational conditions~$\bxi$. In~\eqref{P:lqr}, for instance, it could correspond to relaxing the state constraints and allowing the system to visit higher risk regions of the state space. If damage to the actuators renders the original control problem infeasible, this may be the only course of action to remain operational.

Observe that for~$\bs \equiv \zeros$, \eqref{P:parametrized} solves the worst-case robust control problem~\eqref{P:robust} for~$\delta = 0$. Indeed, if~$\bg(\bz,\bxi) \leq \zeros$ for all~$\bxi \in \calK$, then~$\Pr[\bg(\bz,\bXi) \leq \zeros] = 1$. This formulation is often found in settings where controllers must abide to requirements under specific contingencies, such as security constrained power allocation~\cite{Capitanescu11s}. In the case of resilience, however, the goal is not to obtain solutions for vanishing slacks, but to adjust~$\bs$ to allow constraint violations for disruptions under which the requirements become too stringent for a robust controller to satisfy. Hence, we are typically interested in solving~\eqref{P:parametrized} with~$\bs(\bxi) \succ \zeros$ for some, if not all, disruptions~$\bxi$.

For any predetermined~$\bs$, \eqref{P:parametrized} is a smooth convex problem that can be solved using any of a myriad of existing methods~\cite{Bertsekas15c}. Yet, designing~$\bs$, which ultimately determines the resilient behavior of the controller, can be quite challenging. Even for a moderate number of contingencies~(cardinality of~$\calK$), finding the right requirement to violate and determining by how much to do so for each state of the world is intricate. This problem is only exacerbated as the number of requirements and/or contingencies grows. In Section~\ref{S:resilience_theory}, we propose a principled approach to designing resilient behavior based on trading off the control performance~$P^\star_\text{Re}(\bs)$ and a measure of violation. Before proceeding, however, we derive the dual problem of~\eqref{P:parametrized} and introduce the results from duality theory needed in the remainder of the paper.

\subsection{Dual resilient control}
	\label{S:duality}

Start by associating the dual variable~$\lambda_i \in L_2^+$ with the~$i$-th requirement, where~$L_2^+ = \{\lambda \in L_2 \mid \lambda \geq 0 \text{ a.e.}\}$. Depending on~$\calK$, $\lambda_i$ may be a function or reduce to a (in)finite-dimensional vector. For conciseness, we collect the~$\lambda_i$ in a vector~$\blambda \in \setR_+^m$. Then, define the Lagrangian of~\eqref{P:parametrized} as
\begin{equation}\label{E:lagrangian}
\begin{aligned}
	\calL(\bz, \blambda, \bs) &= J(\bz)
		+ \int_\calK \blambda(\bxi)^T \big[ \bg(\bz, \bxi) - \bs(\bxi) \big] d\bxi
		\text{.}
\end{aligned}
\end{equation}
From the Lagrangian~\eqref{E:lagrangian}, we obtain the dual problem
\begin{prob}[\textup{D-RE}]\label{P:dual_parametrized}
	D^\star_\text{Re}(\bs) = \max_{[\blambda]_i \in L_2^+}
		\min_{\bz \in \setR^p}\ \calL(\bz, \blambda, \bs)
		\text{.}
\end{prob}

Under mild conditions, $D^\star_\text{Re}(\bs)$ attains~$P^\star_\text{Re}(\bs)$ and solving~\eqref{P:dual_parametrized} becomes equivalent to solving~\eqref{P:parametrized}. This fact together with the convexity of~\eqref{P:parametrized} imply that the well-known KKT necessary conditions are also sufficient. In these cases, we obtain a direct relation between the solutions of~\eqref{P:dual_parametrized} and the sensitivity of~$P_\text{Re}$ with respect to~$\bs$. These facts are formalized in Propositions~\ref{T:zdg} and~\ref{T:diff_P}.

\begin{assumption}\label{A:slater}
There exists~$\bzb$ such that~$\bg(\bzb,\bxi) < \zeros$ for all~$\bxi \in \calK$.
\end{assumption}

\begin{proposition}[{\cite[Prop.~5.3.4]{Bertsekas09c}}]\label{T:zdg}
Under Assumptions~\ref{A:convexity} and~\ref{A:slater}, strong duality holds for~\eqref{P:parametrized}, i.e., $P^\star_\text{Re}(\bs) = D^\star_\text{Re}(\bs)$. Moreover,

\begin{enumerate}[(i)]
	\item if~$\blambda^\star(\bs)$ is a solution of~\eqref{P:dual_parametrized}, then~$\bz_\text{Re}^\star(\bs) = \argmin_{\bz \in \setR^p}\ \calL(\bz, \blambda^\star(\bs), \bs)$ is a solution of~\eqref{P:parametrized};
	
	\item if~$\bz^\prime$ is a feasible point of~\eqref{P:parametrized} and~$[\blambda^\prime]_i \in L_2^+$, then~$\bz^\prime$ is the solution of~\eqref{P:parametrized} and~$\blambda^\prime$ is a solution of~\eqref{P:dual_parametrized} if and only if
\begin{subequations}\label{E:kkt}
\begin{align}
	\nabla \calL(\bz^\prime,\blambda^\prime,\bs) &= \zeros
		\label{E:kkt_grad}
	\\
	[\blambda^\prime(\bxi)]_i \left[ g_i(\bz^\prime,\bxi) - s_i(\bxi) \right] &= \zeros
		\text{, for all } \bxi \in \calK
		\label{E:kkt_slackness}
		\text{.}
\end{align}
\end{subequations}

\end{enumerate}
\end{proposition}

\begin{proposition}\label{T:diff_P}
Let~$\blambda^\star$ be a solution of~\eqref{P:dual_parametrized}. Under Assumptions~\ref{A:convexity} and~\ref{A:slater}, it holds that~$\nabla_{\bs} P^\star_\text{Re}(\bs) \big\vert_{\bxi} = -\blambda^\star(\bxi)$.
\end{proposition}

\begin{proof}
This is a direct consequence of~\cite[Thm.~3.2]{Shapiro95d}. The only non-trivial condition is that the solution set of~\eqref{P:parametrized} is \emph{inf-compact}. This stems from the fact that the~$g_i$ are radially unbounded and continuous, in which case the feasible set of~\eqref{P:parametrized} is respectively bounded and closed.
\end{proof}

Having established these duality results, we now introduce a method to design resilient behavior based on compromising between control performance and requirement violations.

\section{RESILIENCE BY COMPROMISE}
	\label{S:resilience_theory}

While straightforward and tractable, the resilient optimal control problem~\eqref{P:parametrized} can lead to a multitude of behaviors, not all of them useful, depending on the choice of slacks. In this section, we take a compromise approach to designing resilient behavior by balancing the control performance~$P^\star_\text{Re}(\bs)$ resulting from the violations~$\bs$ and a measure of the magnitude of this violation.

The rationale behind this compromise is that even after adapting to a disruption, the behavior of the resilient system should remain similar to that of the undisturbed one in at least some aspects. If the specifications of the original problem must be completely replaced, then it was most likely ill-posed to begin with. Still, regardless of the disruption caused by~$\bxi$, increasing violations always improves the control performance. Indeed, $P^\star_\text{Re}$ is a non-increasing function of~$\bs$ in the sense that since the feasible set of~\eqref{P:parametrized} with slacks~$\bs^\prime$ is contained in that of~\eqref{P:parametrized} with slacks~$\bs \preceq \bs^\prime$, it immediately holds that~$P^\star_\text{Re}(\bs^\prime) \leq P^\star_\text{Re}(\bs)$.

Hence, all resilient systems must strike a balance between violating requirements to remain operational~(or improve their performance) and stay close to the original specifications. This balance is naturally mediated by the likelihood of the violation occurring, i.e., on the probability of the operating conditions~$\bxi$, in the sense that larger deviations of the original problem are allowed for less likely disruptions.

Explicitly, associate to each relaxation~$\bs$ a scalar violation cost~$h(\bs)$. Then, the specification~$\bs^\star$ is compromise-resilient if any further requirement violations would improve performance~(reduce control cost) as much as it would increase the violation cost, i.e.,
\begin{equation}\label{E:resilient}
	\nabla P^\star_\text{Re}(\bs) \big\vert_{\bs^\star,\, \bxi}
		= -\nabla h(\bs^\star(\bxi)) f_{\bXi}(\bxi)
		\text{,}
\end{equation}
where~$\nabla h$ is the gradient of~$h$. Without loss of generality, we assume~$h(\zeros) = \zeros$. The existence of the derivative of the optimal value function~$P^\star_\text{Re}$ obtains from Proposition~\ref{T:diff_P}.

\begin{assumption}\label{A:h}
The cost~$h$ is a twice differentiable, strongly convex function.
\end{assumption}

Observe that~$\bs^\star$ need not vanish even if~\eqref{P:parametrized} is feasible for~$\bs \equiv \zeros$. Hence, contrary to robustness from~\eqref{P:robust}, a compromise-resilient system may violate the original requirements even for mild disturbances that would not, in principle, warrant it. Nevertheless, whenever it does, it does so in a controlled and parsimonious manner.

Though obtaining a solution of~\eqref{P:parametrized} under the resilient equilibrium~\eqref{E:resilient} may appear challenging, it is in fact straightforward since it is equivalent to a convex optimization problem~(Section~\ref{S:equivalent_problem}). Hence, the balance~\eqref{E:resilient} induces relaxations that explicitly minimize the expected violation cost. Still, this does not characterize the resilient behavior resulting from~\eqref{E:resilient}. We therefore proceed to quantify the effect of the operational conditions~$\bxi$ on resilient behavior~$\bs$, showing that it identifies and relaxes requirements that are harder to satisfy under each disruption. To conclude, we construct a cost such that the resilience-by-compromise solution from~\eqref{E:resilient} is also a solution of the robust control problem~\eqref{P:robust}. Hence, resilience and robustness effectively optimize different objectives and may, in many applications, both be desired properties.

\subsection{Inverse optimality of resilience by compromise}
	\label{S:equivalent_problem}

Consider the optimization problem
\begin{prob}\label{P:resilient}
	P^\star_\text{Re} = \min_{\substack{\bz \in \setR^p\\ s_i \in L_2^+}}&
		&&J(\bz) + \E\left[ h\big( \bs(\bXi) \big) \right]
	\\
	\subjectto& &&g_i(\bz, \bxi) \leq s_i(\bxi)
		\text{,} \quad \text{for all } \bxi \in \calK
		\text{,}
	\\
	&&&\qquad i = 1,\dots,m
		\text{,}
\end{prob}
where the expectation is taken with respect to the distribution of the random variable~$\bXi$. The solution of~\eqref{P:resilient} is the same as the modified problem~\eqref{P:parametrized} with slacks satisfying the resilient equilibrium~\eqref{E:resilient}.

\begin{proposition}\label{T:equivalent_problem}
Let~$(\bz_\text{Re}^\star,\bs^\star)$ be the solution of~\eqref{P:resilient}. Then, $P^\star_\text{Re} = P^\star_\text{Re}(\bs^\star)$ and~$\bs^\star$ are the unique slacks that satisfy the equilibrium~\eqref{E:resilient}.
\end{proposition}

\begin{proof}
To show~\eqref{P:resilient} is equivalent solving~\eqref{P:parametrized} subject to the compromise~\eqref{E:resilient}, we leverage the fact that the KKT conditions in Proposition~\ref{T:zdg}(ii) are necessary and sufficient for convex programs under Assumption~\ref{A:slater}.

Start by defining the Lagrangian of~\eqref{P:resilient} as
\begin{equation}\label{E:lagrangian_resilient}
\begin{aligned}
	\calL^\prime((\bz,\bs), \bmu)
		&= f_0(\bz) + \E\left[ h\big( \bs(\bXi) \big) \right]
	\\
	{}&+ \int_\calK \bmu(\bxi)^T \big[ \bg(\bz, \bxi) - \bs(\bxi) \big] d\bxi
		\text{,}
\end{aligned}
\end{equation}
where we write~$(\bz,\bs)$ to emphasize that they are both primal variables of~\eqref{P:resilient} as opposed to~\eqref{P:parametrized} in which~$\bz$ is an optimization variable and~$\bs$ is a parameter.

From Proposition~\ref{T:zdg}(ii), if~$(\bz_\text{Re}^\star,\bs^\star)$ is a solution of~\eqref{P:resilient}, then there exists~$\bmu^\star$ such that~$\nabla \calL^\prime((\bz_\text{Re}^\star,\bs^\star), \bmu^\star) = \zeros$ and~$[\bmu^\star(\bxi)]_i \left[ g_i(\bz_\text{Re}^\star,\bxi) - s_i(\bxi) \right] = 0$, for all~$\bxi \in \calK$. Separating the elements of the gradient of~\eqref{E:lagrangian_resilient} for~$\bz$ and~$\bs$, its KKT conditions become
\begin{equation}\label{E:kkt_resilient}
	\nabla_{\bz} \calL(\bz_\text{Re}^\star, \bmu^\star, \bs^\star) = \zeros
	\text{ and }
	\nabla h \big( \bs^\star(\bxi) \big) - \bmu^\star(\bxi) = \zeros
		\text{,}
\end{equation}
where~$\calL$ is the Lagrangian~\eqref{E:lagrangian} of~\eqref{P:parametrized} with slacks~$\bs^\star$. The first equation in~\eqref{E:kkt_resilient} shows that~$\bz_\text{Re}^\star$ is also a solution of~\eqref{P:parametrized} for the slacks~$\bs^\star$. Using Proposition~\ref{T:diff_P}, the second equation shows that~$\bs^\star$ satisfies the equilibrium~\eqref{E:resilient}. The reverse relation holds directly, since the KKT conditions of both problems are actually identical.
\end{proof}

Proposition~\ref{T:equivalent_problem} shows that under the resilience equilibrium~\eqref{E:resilient}, \eqref{P:parametrized} optimizes both the control performance function~$J$ and the expected requirement violation cost. In other words, though the resilient formulation may violate the requirements for most states of the world, it does so in a parsimonious manner.

It is worth noting that relaxing constraints as in~\eqref{P:resilient} is common in convex programming and is used, for instance, in phase~1 solvers for interior-point methods~\cite{Bertsekas15c}. The goal in~\eqref{P:resilient}, however, is notably different. Indeed, resilience does not seek a solution~$\bz^\dagger$ for which the slacks~$\bs(\bxi)$ vanish for all~$\bxi$. Its aim is to adapt to situations in which disruptions are so extreme that only by modifying the underlying control problem is it possible to remain operational. Hence, it seeks~$\bs \succ \zeros$ for some, if not all, disruptions~$\bxi$.

Another consequence of Proposition~\ref{T:equivalent_problem} is that the compromise-resilient control problem~\eqref{P:parametrized}--\eqref{E:resilient} has a straightforward solution since it is equivalent to a convex optimization program, namely~\eqref{P:resilient}. Nevertheless, it turns out that a more efficient algorithm can be obtained by understanding how resilience violates the requirements to respond to disruptions. That is the topic of the next section.

\subsection{Quantifying the effect of disturbances}
	\label{S:counterfactual}

Proposition~\ref{T:equivalent_problem} shows that resilient control minimizes the problem modifications through the cost~$h$. In constrast, the following proposition explicitly describes the effect of a disturbance~$\bxi$ on the violations~$\bs$.

\begin{proposition}\label{T:slacks}
	Let~$\bz_\text{Re}^\star(\bs^\star)$ be the solution of~\eqref{P:parametrized} for the resilient slacks~$\bs^\star$ from~\eqref{E:resilient} and~$\blambda^\star(\bs^\star)$ be the solution of its dual problem~\eqref{P:dual_parametrized}. Then,
	\begin{equation}\label{E:slacks}
		\bs^\star = \left( \nabla h \right)^{-1} \left[ \frac{\blambda^\star(\bs^\star)}{f_{\bXi}} \right]
			\text{.}
	\end{equation}
	\vspace{-4pt}
\end{proposition}

\begin{proof}
Follows by applying Proposition~\ref{T:diff_P} to the equilibrium~\eqref{E:resilient} to obtain~$\blambda^\star(\bs^\star) = \nabla h(\bs^\star) f_{\bXi}$. Recall that the Jacobian of the gradient~$\nabla h$ is the Hessian~$\nabla^2 h$ and that since~$h$ is strongly convex~(Assumption~\ref{A:h}), it holds that~$\nabla^2 h \succ 0$. Immediately, the inverse of the gradient exists by the inverse function theorem, yielding~\eqref{E:slacks}.
\end{proof}

Proposition~\ref{T:slacks} establishes a fixed point relation between the resilient slacks~$\bs^\star$ and the optimal dual variables~$\blambda^\star(\bs)$. This is not surprising in view of the well-known sensitivity interpretation of dual variables for convex programs. Indeed, dual variable represent how much the objective stands to change if a constraint were relaxed or tightened. Given the monotone increasing nature of~$\nabla h$~(due to the strong convexity of~$h$, Assumption~\ref{A:h}), it is clear from~\eqref{E:slacks} that the resilient formulation identifies and relaxes constraints that are harder to satisfy. Hence, if a disruption~$\bxi$ makes it difficult for the resilient system to meet a requirement, it will modify that requirement according to its difficulty. This change is mediated by the variation in the resilience cost~$h$ and the likelihood of the disruption~$f_{\bXi}(\bxi)$, which determine the amount by which the requirement is relaxed.

The choice of~$h$ therefore plays an important role in the resulting resilient behavior. For instance, if the violation cost is linear, i.e., $h(\bs) = \bgamma^T \bs$, $\bgamma \in \setR_+^m$, the equilibrium~\eqref{E:resilient} occurs for~$[\bs^\star]_i = [\bgamma]_i^{-1}$. Hence, the violations are independent of the disruptions and the solution is the same as if~\eqref{P:parametrized} were solved for predetermined slacks. A more interesting phenomenon occurs for quadratic cost structures, e.g., $h(\bs) = \bs^T \bGamma \bs$, for~$\bGamma \succ 0$. Then, the violations are proportional to the dual variables as in~$\bs^\star = \bGamma^{-1} \blambda^\star(\bs^\star) / f_{\bXi}$. In this case, the resilient violations are proportional to the requirement difficulty and inversely proportional to the likelihood of the disruption.

Given this wide range of resilient behaviors, a question that arises is how they relate to those induced by the robust formulation. We explore this question in the sequel by relating the resilient control problem~\eqref{P:resilient} to its robust counterpart~\eqref{P:robust}.

\subsection{Resilience vs.\ robustness}
	\label{S:resilient_vs_robust}

On the surface, the robust~\eqref{P:robust} and resilient~\eqref{P:resilient} control problems are strikingly different. And in fact, it is clear from the discussion in the previous section that depending on the choice of~$h$, their behaviors can be quite dissimilar. Yet, it turns out that~\eqref{P:robust} and~\eqref{P:resilient} are equivalent under mild conditions for an appropriate choice of~$h$, as shown in the following proposition.

\begin{proposition}\label{T:resilient_v_robust}
Let~$\bz_\text{Re}^\dagger$ be a solution of~\eqref{P:resilient} with~$h_\text{Ro}(\bs) = -\gamma \prod_{i=1}^m \big( 1 - \mathbb{H}(s_i) \big)$, where~$\mathbb{H}$ is the Heaviside function, i.e., $\mathbb{H}(x) = 1$ if~$x \geq 0$ and zero otherwise. For each~$\gamma \geq 0$ there exists a~$\delta^\dagger \in [0,1]$ such that~$\bz_\text{Re}^\dagger$ is a solution of~\eqref{P:robust} with probability of failure~$\delta^\dagger$.
\end{proposition}

\begin{proof}
Fix~$\gamma$ in the violation cost~$h_\text{Ro}$ defined in the hypothesis and let~$(\bz_\text{Re}^\dagger, \bs_\text{Re}^\dagger)$ be a solution pair of the resilience-by-compromise problem~\eqref{P:resilient} and~$\bz_\text{Ro}^\star$ be a solution of the robust~\eqref{P:robust} with~$1-\delta^\dagger = \Pr\left[ \bg(\bz_\text{Re}^\dagger, \bXi) \leq 0 \right]$.

Immediately, the value of~\eqref{P:resilient} is achieved for~$\bz = \bz_\text{Re}^\dagger$ and~$\bs = \bs_\text{Re}^\dagger$. What is more, note that the solution pair~$(\bz,\bs) = (\bz_\text{Ro}^\star, \bg(\bz_\text{Ro}^\star,\cdot))$ is trivially feasible for~\eqref{P:resilient} and can therefore be used to upper bound its value as in
\begin{equation}\label{E:rvr1}
\begin{aligned}
	P_\text{Re}^\star &= J(\bz_\text{Re}^\dagger)
		- \gamma \E\left[ \prod_{i = 1}^m
			\left( 1-\mathbb{H} \left( \big[ \bs_\text{Re}^\dagger(\bXi) \big]_i \right) \right)
		\right]
	\\	
	&\leq J(\bz_\text{Ro}^\star)
		- \gamma \E\left[ \prod_{i = 1}^m
			\bigg( 1-\mathbb{H} \left( g_i(\bz_\text{Ro}^\star, \bXi) \right) \bigg)
		\right]
		\text{.}
\end{aligned}
\end{equation}
Due to the form of~$\mathbb{H}$, the expectations in~\eqref{E:rvr1} reduce to probabilities. We then obtain
\begin{multline}\label{E:rvr2}
	J(\bz_\text{Re}^\dagger)
		- \gamma \Pr\left[ \bs^\dagger(\bXi) \leq \zeros \right]
	\\
	{}\leq J(\bz_\text{Ro}^\star)
		- \gamma \Pr\left[ \bg(\bz_\text{Ro}^\star, \bXi) \leq \zeros \right]
		\text{.}
\end{multline}
Since~$\bz_\text{Ro}^\star$ is a solution of~\eqref{P:robust} with probability of failure~$\delta^\dagger$, \eqref{E:rvr2} becomes
\begin{equation}\label{E:rvr3}
\begin{aligned}
	J(\bz_\text{Re}^\dagger)
		- \gamma \Pr\left[ \bs^\dagger(\bXi) \leq 0 \right]
	\leq J(\bz_\text{Ro}^\star)
		- \gamma (1-\delta^\dagger)
		\text{.}
\end{aligned}
\end{equation}

To conclude, recall from~\eqref{P:resilient} that~$\bg(\bz_\text{Re}^\dagger(\gamma), \bxi) \leq \bs_\text{Re}^\dagger(\bxi)$ for all~$\bxi \in \calK$, which by monotonicity of the Lebesgue integral implies that
\begin{equation*}
	\Pr\left[ \bs_\text{Re}^\dagger(\bXi) \leq 0 \right]
		\leq \Pr\left[ \bg(\bz_\text{Re}^\dagger, \bXi) \leq 0 \right]
		= 1-\delta^\dagger
		\text{.}
\end{equation*}
Hence, we obtain from~\eqref{E:rvr3} that~$J(\bz_\text{Re}^\dagger) \leq P_\text{Ro}^\star$. Since~$\bz_\text{Re}^\dagger$ is a feasible point of~\eqref{P:robust} with probability of failure~$\delta^\dagger$ by design and its control performance achieves the optimal value~$P_\text{Ro}^\star$, it must be a solution of~\eqref{P:robust}.
\end{proof}

Proposition~\ref{T:resilient_v_robust} gives conditions on the violation cost~$h$ such that a resilience-by-compromise controller behaves as a robust one. In particular, it states that there exists a fixed, strict violation cost, i.e., one that charges a fixed price only if some requirement is violated, such that resilience by compromise reduces to robustness. This cost essentially determines the level of control performance~$J$ above which the controller chooses to pay~$\gamma$ to give up on satisfying the requirements altogether. Notice that Proposition~\ref{T:resilient_v_robust} holds even though the resulting problem is not convex.

In that sense, resilience can be thought of as a soft version of robustness: whereas the violation magnitude matters for the former, only whether the requirement is violated impacts the latter. For certain critical requirements, this all-or-nothing behavior may be the only acceptable one. In these cases, constraints should be treated as robust with appropriate satisfaction levels. Other engineering requirements, however, are nominal in nature and can be relaxed as long as violations are small and short-lived. Treating these constraints as resilient enables the system to continue operating under disruptions while remaining robust with respect to critical specifications. For instance, if a set of essential requirements needs a level of satisfaction so high that the control problem becomes infeasible, nominal constraints can be adapted to recover a useful level of operation.

By leveraging Proposition~\ref{T:resilient_v_robust}, this can be achieved by posing a control problem that is both robust and resilient. To do so, let~$\calS \subseteq [m]$ be the set of soft~(nominal) requirements, i.e., those that can withstand relaxation, and~$\calH \subseteq [m]$ be the set of hard~(critical) requirements, i.e., those that cannot be violated under any circumstances. Naturally, $\calS \cap \calH = \emptyset$ and~$\calS \cup \calH = [m]$. We can then combine~\eqref{P:robust} and~\eqref{P:resilient} into a single problem, namely
\begin{prob}\label{P:complete}
	\minimize_{\substack{\bz \in \setR^p,\\s_i \in \calF}}&
		&&f_0(\bz) + \E\left[ \sum_{i \in \calS} h_i\big( s_i(\bXi) \big)
			+ \sum_{i \in \calH} h_\text{Ro}\big( s_i(\bXi) \big)\right]
	\\
	\subjectto& &&f_i(\bz, \bxi) \leq s_i(\bxi)
		\text{,} \ \ \forall \bxi \in \calK
		\text{, } i = 1,\dots,m
		\text{.}
\end{prob}

Whereas~\eqref{P:complete} provides a complete solution to designing robust/resilient systems, it is worth noting that it is not a convex optimization problem. What is more, the non-smooth nature of~$\mathbb{H}$ poses a definite challenge to even approximating its solution. Enabling the solution of this general problem is therefore beyond the scope of this paper. Nevertheless, we describe in the sequel an efficient algorithm to tackle resilience-by-compromise by directly solving~\eqref{P:parametrized} for the resilient equilibrium~\eqref{E:resilient}.

\section{A MODIFIED ARROW-HURWICZ ALGORITHM}
	\label{S:algorithm}

\begin{figure}[tb]
\centering
\includesvg[width=\columnwidth]{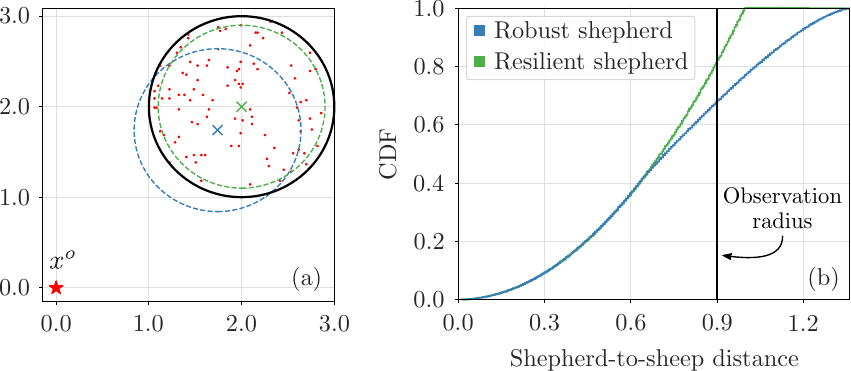}
\vspace{-14pt}
	\caption{Robust and resilient solution to the shepherd problem: (a)~Shepherd plans; (b)~distribution of maximum distance between shepherd and sheep.}
	\label{F:shepherd}
\vspace{-10pt}
\end{figure}

\begin{figure*}
\centering
\includesvg{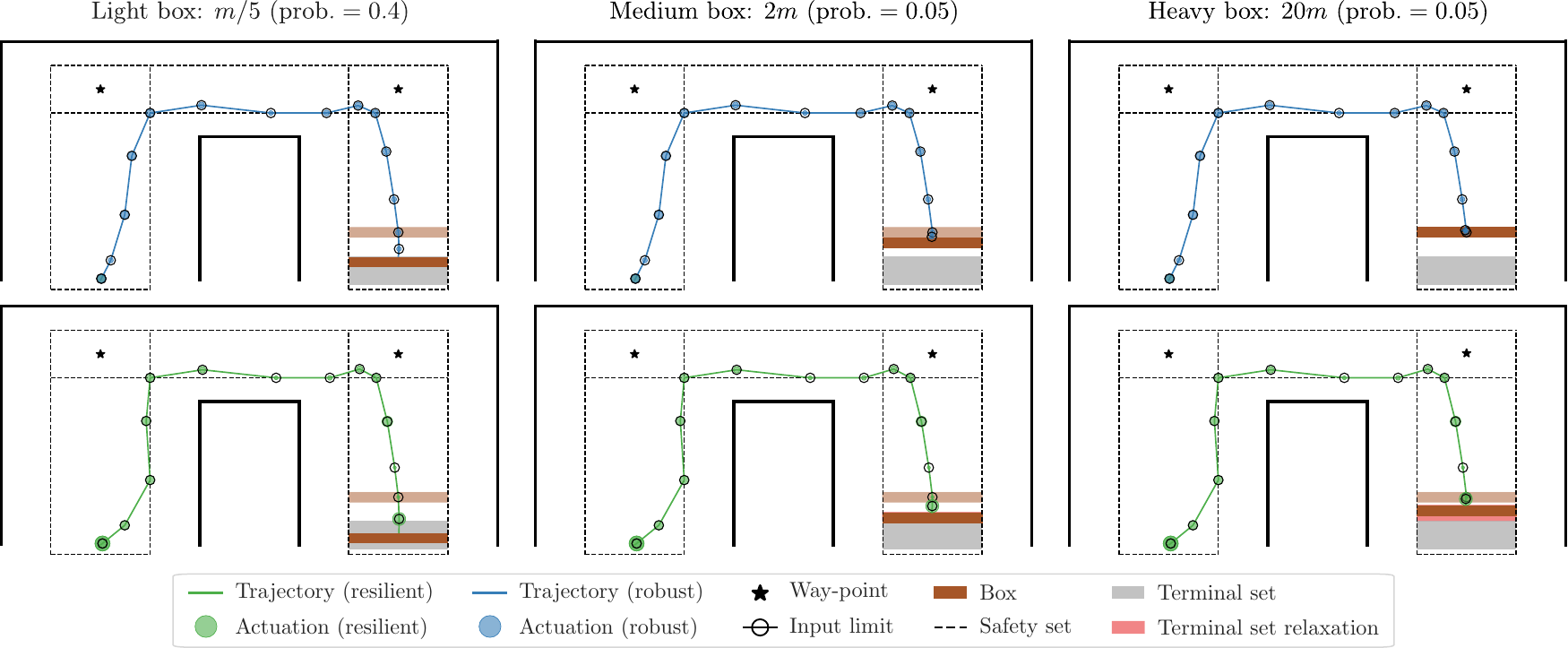}
\vspace{-14pt}
    \caption{Robust and resilient controllers for the quadrotor navigation problem. The radius of the markers are proportional to the actuation strength.}
    \label{F:navigation}
\vspace{-10pt}
\end{figure*}

In view of Proposition~\ref{T:equivalent_problem}, solving the resilient control problem~\eqref{P:parametrized} subject to the equilibrium~\eqref{E:resilient} reduces to obtaining a solution of~\eqref{P:resilient}. Given its a smooth, convex nature, this can be done using any of a myriad of methods~\cite{Bertsekas15c}. One approach that is particularly promising is to use a modified primal-dual algorithm that takes into account the results in Proposition~\ref{T:slacks}.

Explicitly, consider the classical Arrow-Hurwicz algorithm for solving~\eqref{P:parametrized}~\cite{Arrow58s}. This method seeks a points that satisfy the KKT conditions~[Proposition~\ref{T:zdg}(2)] by updating the primal and dual variables using gradients of the Lagrangian~\eqref{E:lagrangian}. Explicitly, $\bz$ is updated by \emph{descending} along the negative gradient of the Lagrangian, i.e.,
\begin{subequations}\label{E:arrow}
\begin{equation}\label{E:primal_arrow}
\begin{aligned}
	\dot{\bz} &= -\nabla_{\bz} \calL(\bz,\blambda,\bs)
	\\
	{}&= - \nabla_{\bz} J(\bz)
		- \int_{\calK} \blambda(\bxi)^T \nabla_{\bz} \bg(\bz,\bxi) d\bxi
		\text{,}
\end{aligned}
\end{equation}
and the dual variables~$\blambda$ are updated by \emph{ascending} along the gradient of the Lagrangian~$\nabla_{\lambda} \calL(\bz,\blambda,\bs)$ using the projected dynamics
\begin{equation}\label{E:dual_arrow}
\begin{aligned}
	\dot{\blambda}(\bxi) &= \Pi_+\big[ \blambda(\bxi),
		\nabla_{\blambda} \calL(\bz,\blambda,\bs) \vert_{\bxi} \big]
	\\
	{}&= \Pi_+\big[ \blambda(\bxi),\, \bg(\bz,\bxi) - \bs(\bxi) \big]
		\text{.}
\end{aligned}
\end{equation}
\end{subequations}
The projection~$\Pi_+$ is introduced to ensure that the Lagrange multipliers remain non-negative and is defined as
\begin{equation}\label{E:projection}
	\Pi_+(\bx,\, \bv) = \lim_{a \to 0} \frac{[\bx + a \bv]_+ - \bx}{a}
		\text{,}
\end{equation}
where~$[\bx]_+ = \argmin_{\by \in \setR_+^m} \norm{\by - \bx}$ is the projection onto the non-negative orthant~\cite{Nagurney12p}.

The main drawback of~\eqref{E:arrow} is that it solves~\eqref{P:parametrized} for a fixed slack~$\bs$ and the desired compromise~$\bs^\star$ in~\eqref{E:resilient} is not known \emph{a priori}. To overcome this limitation, we can use Proposition~\ref{T:slacks} and replace~\eqref{E:dual_arrow} by
\begin{equation}\label{E:alt_dual_arrow}
	\dot{\blambda}(\bxi) = \Pi_+\left[ \blambda(\bxi),\, \bg(\bz,\bxi)
		- \nabla h^{-1}\left( \frac{\blambda(\bxi)}{f_{\bXi}(\bxi)} \right)
	\right]
		\text{.}
\end{equation}
The dynamics~\eqref{E:primal_arrow}--\eqref{E:alt_dual_arrow} can be shown to converge to a point that satisfies the KKT conditions in Proposition~\ref{T:zdg}(2) as well as the equilibrium~\eqref{E:resilient} using an argument similar to~\cite{Cherukuri16a} that relies on classical results on projected dynamical systems~\cite[Thm.~2.5]{Nagurney12p} and the invariance principle for Carath\'{e}odory systems~\cite[Prop.~3]{Bacciotti06n}. Hence, they simultaneously solves three problems by obtaining (i)~requirement violations~$\bs^\star$ that satisfies~\eqref{E:resilient}, (ii)~the solution~$\bz^\star(\bs^\star)$ of~\eqref{P:parametrized} for the violations~$\bs^\star$, and~(iii)~dual variables~$\blambda^\star(\bs^\star)$ that solve~\eqref{P:dual_parametrized} for~$\bs^\star$. Due to space constraints, details of this proof are left for a future version of this work.

\section{NUMERICAL EXPERIMENTS}
	\label{S:sims}

In this section, we illustrate the use of resilient optimal control in two applications: \emph{the shepherd problem}, in which we plan a configuration in order to surveil targets~(Section~\ref{S:shepherd}), and \emph{navigation in partially known environments}, in which a quad-rotor must follow way-points to a target that is behind an obstruction of unknown mass~(Section~\ref{S:navigation}). We also illustrate an online extension of our resilience framework in which a quad-rotor adapts to wind gusts~(Section~\ref{S:online}). Due to space constraints, we only provide brief problem descriptions in the sequel. Details can be found in~\cite{extended}.

\subsection{The shepherd problem}
\label{S:shepherd}

We begin by illustrating the differences between robustness and resilience in a static surveillance planning problem. Suppose an agent~(\emph{the shepherd}) must position itself to supervise a set of targets~(\emph{the sheep}). Without prior knowledge of their position, the shepherd assumes the sheep are distributed uniformly at random within a perimeter of radius~$R$. The surveillance radius~$r$ of the shepherd is enough to cover only~$90\%$ of that area. The shepherd also seeks to minimize its displacement from its home situated at~$\bx^o$. If we let~$\bXi_i$ denote the position of the~$i$-th sheep and~$\calK$ be the ball described by the radius-$R$ perimeter, the robust formulation~\eqref{P:robust} becomes
\begin{prob}
	\minimize_{\bx}& &&\norm{\bx - \bx^o}^2
	\\
	\subjectto& &&\Pr\left[ \norm{\bx- \bXi_i}^2 \leq r^2 \right] \geq 1-\delta
\end{prob}
and the resilient problem~\eqref{P:resilient} yields
\begin{prob}
	\minimize_{\bx}& &&\norm{\bx - \bx^o}^2
		+ \E \left[ \sum_{i = 1}^m s_i^2(\bXi_i) \right]
	\\
	\subjectto& &&\norm{\bx- \bxi_i}^2 \leq r^2 + s_i(\bxi_i)
		\text{,} \quad \bxi_i \in \calK
		\text{.}
\end{prob}

Fig.~\ref{F:shepherd} show results for~$\delta = 0.2$. In order to meet the set probability of failure, the robust moves away from the origin only as much as necessary, leading to a plan that has lower cost than the resilient. The resilient solution, on the other hand, is willing to pay the extra cost to move to the center of the perimeter so that when a sheep steps out of its surveillance radius, it does not go too far~(Fig.~\ref{F:shepherd}b). This example illustrates the difference between robust and resilient planning. While the robust system saves on cost by minimally meeting the specified requirement violation, the resilient system takes into account the magnitude of the violations. Hence, it is willing to pay the extra cost in order to reduce future violations.

\subsection{Way-point navigation in a partially known environment}
\label{S:navigation}

A quadrotor of mass~$m$ must plan control actions to navigate the hallway shown in Fig.~\ref{F:navigation} by going close to the way-points~(stars) at specific instants while remaining within a safe distance of the walls and limiting the maximum input thrust. Between the quadrotor and its target, however, there may exist an obstruction of \emph{a priori} unknown mass~(brown box). This box modifies the dynamics of the quadrotor in a predictable way depending on its mass, i.e., the quadrotor can push the box by applying additional thrust but the magnitude of this thrust is not known beforehand. Since it is not possible to find a set of control actions that is feasible for all obstruction masses, we set~$\delta = 0.1$ for the robust controller. On the other hand, the resilient controller is allowed to relax both thrust limits and the terminal set. Hence, it can choose between actuating harder to push the box or deem it too heavy and stop before entering the room.

Notice that while the robust plan reaches the terminal set for the light obstruction, it is unable to do so in the other two cases. This is to be expected given that it was not designed to do so. The resilient controller, however, displays a smoother degradation as the weight of the obstruction increases. Notice that it chooses which requirement to violate by compromising between their satisfaction and the control objective~(LQR). While it violates the maximum thrust constraint enough to push the medium box almost into the terminal set, it deems the heavy box to not be worth the effort and relaxes the terminal set instead. This leads to a more graceful behavior degradation than the one induced by the robust controller. Moreover, observe that the resilient controller also uses additional actuation in the beginning to more quickly approach the wall and reduce the distance traveled. This is an example of the ``unnecessary yet beneficial'' requirement violations that resilient control may perform in order to improve the control performance. Naturally, if thrust requirements are imposed by hardware limitations, then the robust solution is the only practical one.

\begin{figure}
\centering
\includesvg{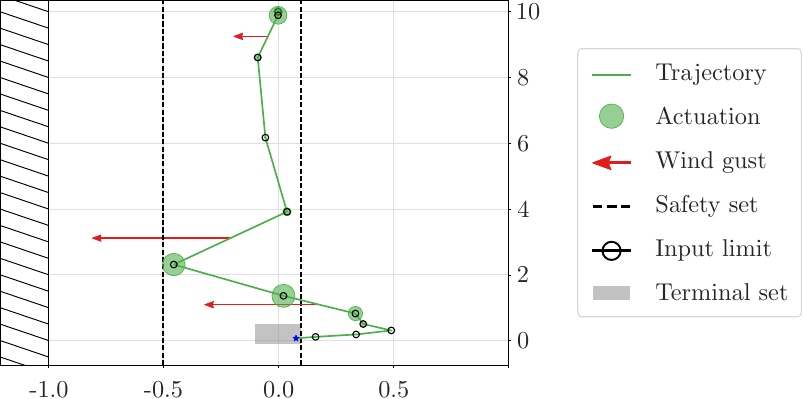}
	\caption{Online resilient control using MPC: quadrotor under wind disruption}
	\label{F:online}
\vspace{-12pt}
\end{figure}

\subsection{Online extension: adapting to wind gusts}
\label{S:online}

The previous examples illustrated the behavior of the resilient optimal control problem formulation introduced in this paper. Another aspect of resilience, beyond planning to mitigate disruptions, is the ability to adapt to disturbances as they occur. This can be achieved by using the resilient optimal control problem in an MPC fashion. We show the result of doing so in Fig.~\ref{F:online}. Here, a quadrotor navigates towards its target~(grey zone) by planning over a~$10$-steps horizon, but executing only the first control action. During this execution, the quadrotor may be hit by an unpredictable wind gust that pushes him towards a wall~(left of the diagram). The quadrotor takes the wind gust suffered into account in its future plan by assuming that the wind will continue to blow at that speed. The resilient controller is allowed to modify the safety set and maximum thrust requirements.

Similar behaviors to Fig.~\ref{F:navigation} can be observed. The resilient controller chooses to violate the thrust constraint in order to pick up speed initially. It does so because the price of using extra actuation is compensated by the improvement in control performance~(LQR). When a gust of wind pushes the quadrotor close to the left boundary of the safety set, it again violates the actuation constraints to stay withing the safe region. It does so in full view that it must now overshoot the safety region on the right. Notice that the resilient behavior of the quadrotor is adaptive: as disruptions occur, the controller plans which requirements should be violated to remain operational. Without these violations, such intense wind gusts would crash the quadrotor into the wall.

\section{CONCLUSION AND FUTURE WORK}

We defined resilient control by embedding control problems with the ability to violate requirements and proposed a method to automatically design these violations by compromising between the control objective and a constraint violation cost. We showed that such a compromise explicitly minimizes changes to the original control problem and that for properly selected costs, robust behaviors can be induced. These results are the first steps toward a resilient control solution capable of adapting to disruptions online. Such behavior can be achieved by combining~\eqref{P:resilient} and MPC as shown in Section~\ref{S:online}. Future works involve analyzing the stability of such solutions and leverage system level synthesis techniques~\cite{Anderson19s} to directly design resilient controllers.

\bibliographystyle{aux_files/IEEEbib}
\bibliography{aux_files/IEEEabrv,aux_files/af,aux_files/bayes,aux_files/control,aux_files/gsp,aux_files/math,aux_files/ml,aux_files/rkhs,aux_files/rl,aux_files/sp,aux_files/stat,aux_files/quadrotor}

\appendix
\section*{Quadrotor model}

\subsection{Dynamics}

In this section, we describe the linearized quadrotor dynamics used in the simulations of Sections~\ref{S:navigation} and~\ref{S:online}. The model is obtained as in~\cite{Beard08q, Sabatino15q} by approximating the quadrotor as a dense sphere of mass~$M$ and radius~$R$ together with four points masses~$m^\prime$ distributed a distance~$l$ from the center of the quadrotor. Hence, the total mass of the quadrotor is given by~$m = M + 4m^\prime$.

The state vector~$\bx \in \setR^{12}$ describes the position and velocities of the quadrotor. Explicitly, we let
\begin{equation}
	\bx = \vect{cccccccccccc}{x & y & z & \phi & \theta & \psi & u & v & w & p & q & r }^T
		\text{,}
\end{equation}
where~$(x,y,z)$ and~$(\phi,\theta,\psi)$ denote the linear and angular positions of the quadrotor, respectively, with respect to a \emph{world} frame oriented as in Fig.~\ref{F:orientation} and~$(u,v,w)$ and~$(p,q,r)$ denote linear and angular velocities in the \emph{body} frame. The control inputs collected in~$\bu \in \setR^4$ are the thrust and net torques along each axis, i.e.,
\begin{equation}
	\bu = \vect{cccc}{f_t & \tau_x & \tau_y & \tau_z}^T
		\text{,}
\end{equation}
which are controlled by adjusting the speed of each of the four propellers. The quadrotor may be operating in windy conditions described by an external disturbance that exert force and torque along each directional axis. This disturbance is represented by the vector~$\bw \in \setR^6$ defined as
\begin{equation}
	\bw = \vect{cccccc}{f_{wx} & f_{wy} & f_{wz} & \tau_{wx} & \tau_{wy} & \tau_{wz}}^T
		\text{,}
\end{equation}
where~$(f_{wx},f_{wy},f_{wz})$ are wind-generated and~$(\tau_{wx},\tau_{wy},\tau_{wz})$ are wind-generated torques along the axis of Fig.~\ref{F:orientation}.

The quadrotor dynamics are inherently non-linear, so we use a linearization around the point
\begin{align*}
	\bxb &= \vect{cccccccccccc}{\bar{x} & \bar{y} & \bar{z} & 0 & 0 & 0 &
		0 & 0 & 0 & 0 & 0 & 0}^T
	\\
	\bub &= \vect{cccc}{ mg & 0 & 0 & 0 }
	\\
	\bwb &= \vect{cccccc}{ 0 & 0 & 0 & 0 & 0 & 0 }
\end{align*}
to obtain the linear system:
\begin{equation}\label{E:dynamics_continuous}
	\dot{\bx} = \bA_c \bx+ \bB_c \bu + \bW_c \bw
\end{equation}
with
\begin{gather*}
	\bA_c = \begin{bmatrix}
		0& 0& 0& 0& 0& 0& 1& 0& 0& 0& 0& 0\\
		0& 0& 0& 0& 0& 0& 0& 1& 0& 0& 0& 0\\
		0& 0& 0& 0& 0& 0& 0& 0& 1& 0& 0& 0\\
		0& 0& 0& 0& 0& 0& 0& 0& 0& 1& 0& 0\\
		0& 0& 0& 0& 0& 0& 0& 0& 0& 0& 1& 0\\
		0& 0& 0& 0& 0& 0& 0& 0& 0& 0& 0& 1\\
		0& 0& 0& 0& -g& 0& 0& 0& 0& 0& 0& 0\\
		0& 0& 0& g& 0& 0& 0& 0& 0& 0& 0& 0\\
		0& 0& 0& 0& 0& 0& 0& 0& 0& 0& 0& 0\\
		0& 0& 0& 0& 0& 0& 0& 0& 0& 0& 0& 0\\
		0& 0& 0& 0& 0& 0& 0& 0& 0& 0& 0& 0\\
		0& 0& 0& 0& 0& 0& 0& 0& 0& 0& 0& 0
	\end{bmatrix}
    \\
	\bB_c = \begin{bmatrix}
		0& 0& 0& 0\\
		0& 0& 0& 0\\
		0& 0& 0& 0\\
		0& 0& 0& 0\\
		0& 0& 0& 0\\
		0& 0& 0& 0\\
		0& 0& 0& 0\\
		0& 0& 0& 0\\
		1/m& 0& 0& 0\\
		0& 1/I_x& 0& 0\\
		0& 0& 1/I_y& 0\\
		0& 0& 0& 1/I_z\\
	\end{bmatrix}
	\\
	\bW_c = \begin{bmatrix}
		0& 0& 0& 0& 0& 0\\
		0& 0& 0& 0& 0& 0\\
		0& 0& 0& 0& 0& 0\\
		0& 0& 0& 0& 0& 0\\
		0& 0& 0& 0& 0& 0\\
		0& 0& 0& 0& 0& 0\\
		1/m& 0& 0& 0& 0& 0\\
		0& 1 /m& 0& 0& 0& 0\\
		0& 0& 1/m& 0& 0& 0\\
		0& 0& 0& 1 /I_x& 0& 0\\
		0& 0& 0& 0& 1 /I_y& 0\\
		0& 0& 0& 0& 0& 1 /I_z
	\end{bmatrix}
\end{gather*}
where~$(I_x,I_y,I_z)$ are the moments of inertia along the~$x$, $y$, and~$z$ axes respectively.

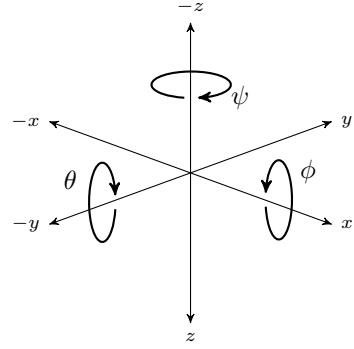
\begin{figure}[tb]
\centering
\begin{tikzpicture}
	[axis/.style={thin, black, ->, >=stealth'}]

	\draw[axis] (0,0) -- (0,2) node [above, black] {\scriptsize $-z$};
	\draw[axis] (0,0) -- (0,-2) node [below, black] {\scriptsize $z$};
	\draw[axis] (0,0) -- (-20:2) node [right, black] {\scriptsize $x$};
	\draw[axis] (0,0) -- (160:2) node [left, black] {\scriptsize $-x$};
	\draw[axis] (0,0) -- (20:2) node [right, black] {\scriptsize $y$};
	\draw[axis] (0,0) -- (200:2) node [left, black] {\scriptsize $-y$};

	\draw [thick, <-, >=stealth'] (0.1,1) arc (-80:260:15pt and 5pt) node [right=14pt, black] {$\psi$};
	\draw [thick, ->, >=stealth'] (1,-0.45) arc (-170:170:5pt and 15pt) node [right=16pt, above=0pt, black] {$\phi$};
	\draw [thick, <-, >=stealth'] (-1,-0.3) arc (10:350:5pt and 15pt) node [left=17pt, above=4pt, black] {$\theta$};
\end{tikzpicture}
	\caption{Orientations and rotations in the world frame.}
	\label{F:orientation}
\end{figure}

Finally, we discretize~\eqref{E:dynamics_continuous} using the sampling time~$T_s$ to obtain the discrete-time system:
\begin{equation}\label{E:dynamics_discrete}
	\bx_{k+1} = \bA \bx_{k} + \bB \bu_{k} + \bW \bw_{k}
		\text{,}
\end{equation}
for
\begin{align*}
	\bA &= e^{\bA_c T_s}
		\text{,}
	\\
	\bB &= \left(\int_{\tau = 0}^{T_s} e^{\bA_c\tau}d\tau \right) \bB_c
		\text{,}
	\\
	\bW &= \left(\int_{\tau = 0}^{T_s} e^{\bA_c\tau}d\tau \right) \bW_c
		\text{.}
\end{align*}

\subsection{Quadrotor parameters}

We use parameters that mimic the Ascending Technologies Hummingbird as in~\cite{Mueller14s}. The total mass of the quadrotor is~$m = 0.5$kg, the distance from center of mass to each propeller is~$0.17$m, and the angular moments are~$I_z = 5.5 \times 10^{-3} \text{kg m}^2$ and~$I_x = I_y = 3.2 \times 10^{-3} \text{kg m}^2$. In order to update these values when the mass of the quadrotor changes~(e.g., when pushing obstacles), we take~$R = 0.0812$m with mass~$M = 0.341$kg for sphere modeling the body and points masses of~$m^\prime = 0.0398$kg to model the propellers.

\section*{Experimental details}
\setcounter{subsection}{0}

\subsection{Way-point navigation in a partially known environment}

In this scenario the quadrotor navigates indoors, so we assume that there is no wind disturbance and take~$\bw_k = \zeros$ for all~$k$.
 
In this scenario, the quadrotor encounters an obstacle of mass~$\Delta$ at time step~$\ell$. We assume an inelastic collision with the stationary obstruction. Hence, iteration~$\ell-1$, the state propogation $\bx_{\ell} = \bA\bx_{\ell-1} + \bB\bu_{\ell-1}$ occurs fully and at iteration~$\ell$ the quadrotor mass instantaneously increases to~$m + \Delta$. By the conservation of momentum, we obtain
\begin{align*}
	u^{\Delta}_\ell &= \frac{m}{m + \Delta} u_\ell
	\\
	v^{\Delta}_\ell &= \frac{m}{m + \Delta} v_\ell
	\\
	w^{\Delta}_\ell &= \frac{m}{m + \Delta} w_\ell
\end{align*}
We assume that the new mass is added to the center of mass of the drone, i.e., $m^{\Delta} = m + \Delta$ and~$M^{\Delta} = M + \Delta$, allowing us to recompute its moments as
\begin{align*}
	I_x^{\Delta} &= I_y^{\Delta} = \frac{2MR^2}{5} + 2l^2m^\prime
		\text{,}
	\\
	I_z^{\Delta} &= \frac{2MR^2}{5} + 4l^2m^\prime
		\text{.}
\end{align*}
These changes affect the input matrix~$\bB$ in the dynamics~\eqref{E:dynamics_discrete}, which become~$\bB^{\Delta}$. We consider the potential changes described in Figure~\ref{F:navigation}, i.e., $\Delta = 0$ with probability~$0.5$, $\Delta = 0.1$kg with probability $0.4$, $\Delta = 1$kg with probability~$0.05$, and~$\Delta = 10$kg with probability~$0.05$.

When formulating our optimal control problems, we use the LQR objective as in~\eqref{P:lqr} with~$\bQ = \bI$, $\bR = \bI$, and for~$\bP$ the solution of the discrete algebraic Ricatti equation, i.e., the cost of the unconstrained infinite horizon LQR, as is typical for MPC problem~\cite{Borrelli17p}. We constrain the control inputs to be in the set~$\calU = \{\bu \in \setR^4 \mid \norm{\bu}_\infty \leq 0.005 \}$. So that our linearization yields a good approximation, we also consider the state constraint~$-\pi/9 \leq \phi,\theta \leq \pi/9$ and~$-\pi \leq \psi \leq \pi$. Additionally, we impose enforce the quadrotor to stay within a safety set so that it never flies closer to~$1$m from the walls of the hallway and at the reduced altitude range of~$4 \leq z \leq 6$m. We combine all these constraints in the set~$\calX$. The terminal set~$\calX_N$ follows the constraints already impose on the angular positions and additionally requires that all velocities~$(u,v,w,p,q,r)$ be within~$[-0.1,0.1]$ and that the linear position~$(x,y,z) \in [-0.1,1] \times [-0.1,0.5] \times [-0.1,0.1]$~(gray region in Fig.~\ref{F:navigation}). The quadrotor starts stationary at~$(x,y,z,\phi,\theta,\psi) = (0,-6,5,0,0,\pi/2)$ and must plan to be close to the waypoints marked with stars~(dashed boxes) at instants~$k = 5$ then~$k = 10$, to hit the~(possible) obstacle at instant~$k = 13$, and be in the terminal set for~$k = N = 15$.

The robust control problem~\eqref{P:robust} solved in this scenario for~$\delta = 0.1$ is given by
\begin{prob*}
	\minimize_{\bx_k,\,\bu_k}&
		&&\bx_N^T \bP \bx_N + \sum_{k = 0}^{N-1} \bx_k^T \bQ \bx_k + \bu_k^T \bR \bu_k 
	\\
	\subjectto& &&\Pr\left[ \bu_k \in \calU \right] \geq 1-\delta
		\text{,} \quad \Pr\left[ \bx_k \in \calX \right] \geq 1-\delta
	\\
	&&&\Pr\left[ \bx_{N} \in \calX_N \right] \geq 1-\delta
	\\
	&&&\bx_{k+1} = \bA \bx_k + \bB \bu_k \text{,} \quad k = 0, \dots, \ell-2
	\\
	&&&\bxh_{\ell} = \bA \bx_{\ell-1} + \bB \bu_{\ell-1}
	\\
	&&&[\bx_{\ell}]_{\calC} = [\bxh_{\ell}]_{\calC}
		\text{,} \quad
		[\bx_{\ell}]_{\bar{\calC}} = \frac{m}{m + m_j} [\bx_{\ell}]_{\bar{\calC}}
		\text{,}
	\\
	&&&\bx_{k+1} = \bA \bx_k + \bB^{\Delta} \bu_k \text{,} \quad k = \ell, \dots, N-1
\end{prob*}
where~$\calC = \{1,2,3,4,5,6,10,11,12\}$ pick out the entries of the state vector corresponding to~$(x,y,z,\phi,\theta,\psi,p,q,r)$ which are conserved across the shock with the obstacle and~$\bar{\calC} = \{7,8,9\}$ pick out the entries of the state vector corresponding to~$(u,v,w)$. Note that due to the form of the disturbance~$\Delta$, the chance constraints can be done without by simply solving the problem simultaneously for~$\Delta = \{0,0.1\}$.

The resilient control problem is posed similarly but using the formulation in~\eqref{P:resilient} with~$h(\bs) = \norm{\bs}^2$:
\begin{prob*}
	\minimize_{\bx_k,\,\bu_k,\,\bs_{u,k},\,\bs_{x}}&
		&&\bx_N^T \bP \bx_N + \sum_{k = 0}^{N-1} \bx_k^T \bQ \bx_k + \bu_k^T \bR \bu_k 
	\\
	&& {}&+ \E\left[ \norm{\bs_{x}(\Delta)}^2
		+ \sum_{k = 0}^{N-1} \norm{\bs_{u,k}(\Delta)}^2 \right]
	\\
	\subjectto& &&\bu_k - \bs_{u,k}(\Delta) \in \calU
		\text{,} \quad \Delta = \{0,0.1,1,10\}
		\text{,}
	\\
	&&&\bx_{N} - \bs_{x}(\Delta) \in \calX_N
	\text{,} \quad \Delta = \{0,0.1,1,10\}
		\text{,}
	\\
	&&&\bx_k \in \calX
	\\
	&&&\bx_{k+1} = \bA \bx_k + \bB \bu_k \text{,} \quad k = 0, \dots, \ell-2
	\\
	&&&\bxh_{\ell} = \bA \bx_{\ell-1} + \bB \bu_{\ell-1}
	\\
	&&&[\bx_{\ell}]_{\calC} = [\bxh_{\ell}]_{\calC}
		\text{,} \quad
		[\bx_{\ell}]_{\bar{\calC}} = \frac{m}{m + \Delta} [\bx_{\ell}]_{\bar{\calC}}
		\text{,}
	\\
	&&&\bx_{k+1} = \bA \bx_k + \bB^{\Delta} \bu_k \text{,} \quad k = \ell, \dots, N-1
\end{prob*}

\subsection{Online extension: adapting to wind gusts}

Throughout this scenario, the mass of the drone remains constant, but we consider external disturbances by taking~$\bw_k \neq \zeros$ in~\eqref{E:dynamics_discrete}. The adaptation is performed by solving the resilient problem~\eqref{P:resilient} in an MPC fashion using observed disturbances. At each time-step~$t$, the autonomous agent uses the wind intensity realized at~$t-1$ and its current state~$\bx(t)$ to plan its actions over a~$N = 10$-steps horizon and then proceeds to execute the first action. Observing the wind disturbance suffered at this new instant~$t$, the agent then starts planning for his next step in a similar manner.

The dynamics used for the planning are those described in~\eqref{E:dynamics_discrete}. The LQR costs and the control input constraint set~$\calU$ are the same as in the previous simulation. So that our linearization yields a good approximation, we consider the angular state constraint~$-\pi/9 \leq \phi,\theta \leq \pi/9$ and~$-\pi \leq \psi \leq \pi$ and limit linear and angular velocities to the range~$[-10,10]$. For safety, the quadrotor must fly within~$(x,y,z) \in [-10,0.1] \times [-0.5,10.1] \times [4,6]$~(within the dashed lines in Fig.~\ref{F:online}), although the resilient controller is allowed to modify this set. We collect these definitions in the set~$\calX$. The terminal set~$\calX_N$~(shown in grey in Fig.~\ref{F:online}) follows the constraints already impose on the angular positions and additionally requires that all velocities~$(u,v,w,p,q,r)$ be within~$[-0.1,0.1]$ and that the linear position~$(x,y,z) \in [-0.1,0.1] \times [-0.1,0.1] \times [-0.1,0.1]$~(gray region in Fig.~\ref{F:navigation}). The quadrotor starts stationary from~$(x,y,z,\phi,\theta,\psi) = (0,10,5,0,0,-\pi/2)$.

The planning of the quadrotor actions is performed using the resilient control problem
\begin{prob*}
	\minimize_{\substack{\bx_k, \bu_k, \bs_k}}&
		&&\bx_N^T \bP \bx_N + \sum_{k = 0}^{N-1} \bx_k^T \bQ \bx_k + \bu_k^T \bR \bu_k 
	\\
	&& {}&+ \sum_{k = 0}^{N-1} \norm{\bs_{x,k}}^2 + \norm{\bs_{u,k}}^2
	\\
	\subjectto& &&\bu_k - \bs_{u,k} \in \calU
		\text{,}
	\\
	&&&\bx_{k} - \bs_{x,k} \in \calX
		\text{,}
	\\
	&&&\bx_N \in \calX_N
	\\
	&&&\bx_{k+1} = \bA \bx_k + \bB \bu_k + \bW \bw
	\text{.}
	\\
	&&&\bx_0 = \bx(t)
		\text{, } \quad \bw = \bw(t-1)
\end{prob*}
where~$\bw(t-1)$ is the wind intensity suffered during the previous time step. Wind gusts are simulated by taking~$f_{wx} = 0.1$ at time step~$t = 2$, $f_{wx} = 0.6$ at time step~$t = 5$, and~$f_{wx} = 0.5$ at time step~$t = 7$. Observe that the autonomous agent does not know the true value of the disturbances before they occur. It relies solely on observations in a model-free manner.

\end{document}

%% file: mysymbol.tex
\DeclareMathOperator{\E}{\mathbb{E}}

\newcommand{\vect}[2]{\ensuremath{[\begin{array}{#1} #2 \end{array}]}}
\newcommand{\norm}[1]{\ensuremath{\left\| #1 \right\|}}
\newcommand{\abs}[1]{\ensuremath{{\left\vert #1 \right\vert}}}

\DeclareMathOperator*{\argmin}{argmin}

\DeclareMathOperator*{\minimize}{minimize}

\DeclareMathOperator{\subjectto}{subject\ to}

\newcommand{\calC}{\ensuremath{\mathcal{C}}}

\newcommand{\calF}{\ensuremath{\mathcal{F}}}

\newcommand{\calH}{\ensuremath{\mathcal{H}}}

\newcommand{\calK}{\ensuremath{\mathcal{K}}}
\newcommand{\calL}{\ensuremath{\mathcal{L}}}

\newcommand{\calS}{\ensuremath{\mathcal{S}}}

\newcommand{\calU}{\ensuremath{\mathcal{U}}}

\newcommand{\calX}{\ensuremath{\mathcal{X}}}

\newcommand{\bA}{\ensuremath{\bm{A}}}
\newcommand{\bB}{\ensuremath{\bm{B}}}

\newcommand{\bI}{\ensuremath{\bm{I}}}

\newcommand{\bP}{\ensuremath{\bm{P}}}
\newcommand{\bQ}{\ensuremath{\bm{Q}}}
\newcommand{\bR}{\ensuremath{\bm{R}}}

\newcommand{\bW}{\ensuremath{\bm{W}}}

\newcommand{\bGamma}{\ensuremath{\bm{\Gamma}}}

\newcommand{\bXi}{\ensuremath{\bm{\Xi}}}

\newcommand{\bg}{\ensuremath{\bm{g}}}

\newcommand{\bs}{\ensuremath{\bm{s}}}

\newcommand{\bu}{\ensuremath{\bm{u}}}
\newcommand{\bv}{\ensuremath{\bm{v}}}
\newcommand{\bw}{\ensuremath{\bm{w}}}
\newcommand{\bx}{\ensuremath{\bm{x}}}
\newcommand{\by}{\ensuremath{\bm{y}}}
\newcommand{\bz}{\ensuremath{\bm{z}}}

\newcommand{\bgamma}{\ensuremath{\bm{\gamma}}}

\newcommand{\blambda}{\ensuremath{\bm{\lambda}}}
\newcommand{\bmu}{\ensuremath{\bm{\mu}}}

\newcommand{\bxi}{\ensuremath{\bm{\xi}}}

\newcommand{\bbR}{\ensuremath{\mathbb{R}}}

\newcommand{\bub}{\ensuremath{\bm{{\bar{u}}}}}

\newcommand{\bwb}{\ensuremath{\bm{{\bar{w}}}}}
\newcommand{\bxb}{\ensuremath{\bm{{\bar{x}}}}}

\newcommand{\bzb}{\ensuremath{\bm{{\bar{z}}}}}

\newcommand{\bxh}{\ensuremath{\bm{{\hat{x}}}}}

\newcommand{\bzh}{\ensuremath{\bm{{\hat{z}}}}}

\def\st/{\textsuperscript{st}}
\def\nd/{\textsuperscript{nd}}
\def\rd/{\textsuperscript{rd}}
\def\th/{\textsuperscript{th}}

\newcommand{\setR}{\bbR}

\newcommand{\zeros}{\ensuremath{\bm{0}}}

\def\nnil{\nil}
\newcounter{prob}
\newenvironment{prob}[1][\nil]{%
	\def\tmp{#1}
	\equation
	\ifx\tmp\nnil
		\refstepcounter{prob}
		\tag{P\Roman{prob}}
	\else
		\tag{\tmp}
	\fi
	\aligned%
}{%
	\endaligned\endequation%
}

\newenvironment{prob*}{%
	\csname equation*\endcsname%
	\aligned%
}{%
	\endaligned%
	\csname endequation*\endcsname%
}